\documentclass[11pt,oneside]{article}
\usepackage[utf8]{inputenc}
\usepackage[english]{babel}
\usepackage{amsmath}
\usepackage{amsfonts}
\usepackage{amssymb}
\usepackage{amsthm}
\usepackage{tikz-cd} 
\usepackage{hyperref}
\usepackage{pgfplots}
\usepackage{subcaption}
\usetikzlibrary{calc}
\usepackage{hyperref}

\usepackage[left=3cm,right=3cm,top=3cm,bottom=3cm]{geometry}
\usepackage[title]{appendix}
\usepackage[]{algorithm2e}

\usepackage[normalem]{ulem}
\useunder{\uline}{\ul}{}
\theoremstyle{plain}
\newtheorem{teo}{}[section]

\newtheorem{cor}[teo]{Corollary}

\newtheorem{lem}[teo]{Lemma}
\newtheorem{thm}[teo]{Theorem}
\theoremstyle{definition}
\newtheorem{ex}[teo]{Example}

\newcommand\blfootnote[1]{%
  \begingroup
  \renewcommand\thefootnote{}\footnote{#1}%
  \addtocounter{footnote}{-1}%
  \endgroup
}

\definecolor{cof}{RGB}{219,144,71}
\definecolor{pur}{RGB}{186,146,162}
\definecolor{greeo}{RGB}{91,173,69}
\definecolor{greet}{RGB}{52,111,72}

\usepackage{graphicx}

\title{From samples to shape: a finite approach to topological invariants}
\author{Pedro J. Chocano}
\date{}

\begin{document}
\maketitle

\begin{abstract}
Given a metric space $Y$ and a compact subspace $X$, we present a method for constructing an inverse sequence of finite topological spaces by sampling points from neighborhoods of $X$ in $Y$. We show that these inverse sequences can be employed to define shape theory and to approximate shape invariants, such as \v{C}ech homology groups. 
\end{abstract}
\blfootnote{This research is partially supported by Grant PID2021-126124NB-I00 from Ministerio de Ciencia, Innovación y Universidades (Spain).}
\section{Introduction}\label{sec:introduccion}

The reconstruction of topological spaces, or their properties, using simpler spaces is a longstanding theme in mathematics. A classical example is the nerve theorem, which asserts that if a topological spaces $X$ admits a ``good'' cover, then the associated simplicial complex $N$ (the nerve of the cover) has the same homotopy type as $X$ (see \cite[Corollary 4G.3]{hatcher2000algebraic}). 

In recent years, attention has shifted toward reconstructing and approximating algebraic invariants of compact metric spaces using finite topological spaces (or more broadly Alexandroff spaces). Several results have demonstrated how inverse sequences of finite topological spaces can be constructed to recover topological properties of a space $X$ via their inverse limits (see \cite{moron2008connectedness,clader2009inverse, mondejar2015hyperspaces,bilski2017inverselimit,mondejar2020reconstruction,chocano2020computational,wen2025mappingspace} for a chronological overview and references therein).

These techniques have also facilitated the development of shape theory through inverse sequences of finite topological spaces \cite{chocano2025shape}, as well as other related aspects \cite{chocano2025categoria}. Furthermore, connections between finiteness and shape theory have been previously explored in works such as \cite{sanjurjo1992intrinsic, sanjurjo1997density, moron2001finite} (listed chronologically).

Given a compact metric space $X$, the constructions of inverse sequences of finite topological spaces in \cite{moron2008connectedness, chocano2020computational, clader2009inverse, mondejar2020reconstruction} share a common feature: they rely on a uniform ``sampling'' of points from $X$. This implicitly assumes some prior knowledge of $X$, which may be unrealistic--even for a simple example like the circle. For example, when sampling points on the unit circle $S^1 = \{(x,y) \in \mathbb{R}^2 \mid x^2 + y^2 = 1\}$ in $\mathbb{R}^2$, a computer may only obtain approximations near $S^1$, rather than exact points on it.

The aim of this paper is to extend previous results by addressing this limitation. Specifically, we propose a method that does not require sampling points directly from $X$, but instead allows for the inclusion of nearby points--potentially outside $X$. This flexibility introduces robustness: even when ``mistakes'' are made in sampling (i.e., selecting points not in $X$), the method still reconstructs topological properties of $X$. Here, we interpret ``mistakes'' as sampling points from a neighborhood of $X$ in a larger metric space $Y$, such as $\mathbb{R}^n$. 

Throughout this manuscript, we assume that $Y$ is a metric space (possibly $\mathbb{R}^n$ with the Euclidean distance), and that $X \subset Y$ is a compact subspace. We consider samples taken both near and within $X$ in order to construct an inverse sequence of finite topological spaces, namely

\[
X_1 \xleftarrow{q_{1,2}}
X_2 \xleftarrow{q_{2,3}}
X_3 \xleftarrow{q_{3,4}}
\cdots
\xleftarrow{q_{n-1,n}}
X_n \xleftarrow{q_{n,n+1}}
X_{n+1} \xleftarrow{q_{n+1,n+2}}
\cdots,
\]

where each map $q_{i,i+1} : X_{i+1} \to X_i$ is continuous for every $i \in \mathbb{N}$. We will refer to this inverse sequence simply as $(X_n, q_{n,n+1})$.

This sequence can be used to compute various shape invariants of $X$, including \v{C}ech homology and cohomology groups, shape dimension, and shape groups. In particular, we show that the inverse sequences constructed here can be used to develop shape theory, obtaining a result analogous to that of \cite{chocano2025shape}. Recall that in \cite{chocano2025shape}, for any compact metric space $X$, and by considering points inside it, the authors constructed an inverse sequence of finite topological spaces that is ``unique'' in a certain sense. These inverse sequences are then used to define morphisms between compact metric spaces, yielding a category that is isomorphic to the shape category of compact metric spaces.

While the result of \cite{chocano2025shape} provides an ``intrinsic'' description of the shape category using finite spaces, our approach follows the spirit of K.\ Borsuk's classical framework (see, e.g., \cite[Appendix~2]{mardevsic1982shape} or \cite{borsuk1971shape}), adapted to the setting of finite spaces. In that approach, the idea is to embed any compact metric space into the Hilbert cube $Q$ and then consider systems of neighborhoods of $X$ inside $Q$ that approximate $X$. A schematic illustration of this can be seen in Figure~\ref{fig_esquema_forma_borsuk}. In black we depict the compact metric space $X$ embedded in $Q$. The first elements of the neighborhood system, denoted $U_1, U_2,$ and $U_3$, are highlighted in different colors. This yields an inverse sequence $(U_n, j_{n,n+1})$, where $j_{i,i+1} : U_{i+1} \to U_i$ is simply the inclusion. The sets $A_i$ chosen inside each $U_i$ can be viewed as discretizations of the neighborhoods $U_i$. Our idea is to construct a finite topological space $U_i(A_i)$ and from it an inverse sequence that mirrors the spirit of this neighborhood-based approach to shape theory. Hence, this can be interpreted as a discretization of Borsuk's classical method.

\begin{figure}[h!]
\centering
\includegraphics[scale=1]{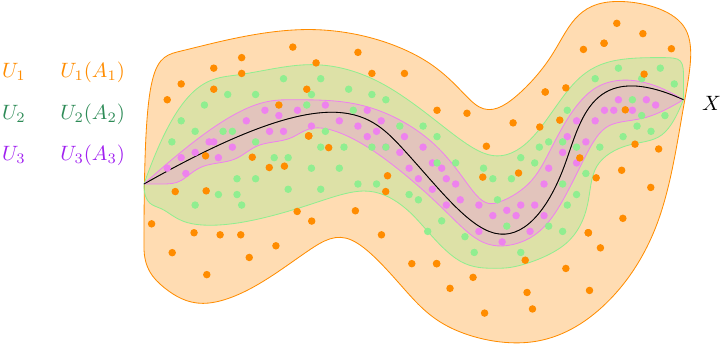}
\caption{Schematic drawing of the classical approach of shape theory and its discretization.}\label{fig_esquema_forma_borsuk}
\end{figure}

It is important to note that alternative approaches to reconstructing algebraic properties exist that do not rely on finite topological spaces. A notable example is the framework presented in \cite{Niyogi2008finding}, where the focus is on Riemannian submanifolds of Euclidean spaces. In that setting, sampling is performed according to a uniform probability distribution, and the analysis leverages geometric tools such as tubular neighborhoods to handle samples taken near the submanifold. These methods are probabilistic in nature and differ fundamentally from the finite-space techniques explored in this paper. Another significant contribution is found in \cite{Cohen2007stability}, where the authors address the problem of estimating the homology groups of a closed subset $X \subset M$, with $M$ a metric space, from a set of possibly inaccurate point samples. Their approach emphasizes stability under perturbations and provides theoretical guarantees for homology inference, offering a complementary perspective to the methods developed here.

The paper is organized as follows. In Section~\ref{Section:preliminaries}, we provide the necessary background and references. Section~\ref{sec_main_result} presents and proves the main results, with the proof of Theorem~\ref{thm:truncated_normal} deferred to Section~\ref{sec_proof} due to its technical nature. In Section~\ref{sec_algorithm_examples}, we outline a method for approximating the Betti numbers of $X$, where $X$ may be an algebraic variety. We conclude with three illustrative examples: two algebraic curves in $\mathbb{R}^2$, and a classical example from shape theory.

\section{Preliminaries}\label{Section:preliminaries}

We recall basic concepts and results from the literature that will be used in subsequent sections. This section is divided into three subsections based on their topics.

\subsection{Finite Topological Spaces}
For a comprehensive introduction to this subject, we strongly recommend consulting the works \cite{barmak2011algebraic,may1966finite}. Let us recall the basic concepts and results from the theory of finite topological spaces. Given a finite topological $T_0$-space $X$ and $x\in X$, consider $U_x$ ($F_x$) as the intersection of every open (closed) set containing $x$. Since $X$ is finite, $U_x$ ($F_x$) is an open (closed) set. Declare $x\leq y$ if and only if $U_x\subseteq U_y$. It is an easy exercise to prove that $(X,\leq)$ is a partially ordered set. On the other hand, for any partially ordered set $(X,\leq)$ (or poset for short) consider the set of lower sets $\tau$ (recall that $L\subseteq X$ is a lower set if $y\leq x$ and $x\in L$ implies that $y\in L$). It is easy to deduce that $\tau$ is a basis for a $T_0$-topology on $X$. These two relations are mutually inverse and this gives that for any finite set $X$ the $T_0$-topologies and the partial orders that can be defined on $X$ are in bijective correspondence. Let $f:X\rightarrow Y$ be a map between finite sets. Then $f$ is a continuous map if and only if $f$ is order-preserving. From this, we can deduce that the category of finite topological $T_0$-spaces and the category of partially ordered sets are isomorphic. From now on, we will not distinguish between a finite partially ordered set and a finite topological $T_0$-space. Moreover, we will assume that every finite topological space is a $T_0$-space and refer to them just as finite topological spaces. 

Given a finite topological space $X$, the \emph{height} of $X$, denoted by $ht(X)$, is the maximum size of a chain in $X$. Similarly, for any $x\in X$, the height $ht(x)$ of a point $x$ is defined as $ht(x)=ht(U_x)$. The Hasse diagram of $X$ is a digraph defined as follows. The vertex set is $X$ and there is an edge from $x$ to $y$ if and only if $x<y$ and there is no $z\in X$ with $x<z<y$. When the last condition holds we simply write $x\prec y$. In subsequent Hasse diagrams we will omit the orientation of the edges and assume an upward orientation.

Let $f,g:X\rightarrow Y$ be two continuous maps between finite topological spaces. We write $f\leq g$ if and only if $f(x)\leq g(x)$ for every $x\in X$.

\begin{thm}\label{thm_homotopia_finitos} Let $f,g:X\rightarrow Y$ be two continuous maps between finite topological spaces. Then $f$ is homotopic to $g$ if and only if there exists a sequence $f_i:X\rightarrow Y$ of continuous maps with $i=1,...,n$ such that $f_1=f$, $f_n=g$ and $f_i\leq f_{i+1}$ or $f_{i+1}\leq f_i$ for every $i=1,...,n-1$.
\end{thm}

For any finite topological space $X$, the \emph{order complex} of $X$, denoted by $\mathcal{K}(X)$, is the simplicial complex whose simplices are the non-empty chains of $X$. Here, $|\mathcal{K}(X)|$ stands for the geometric realization of $\mathcal{K}(X)$. Consider $f_X:|\mathcal{K}(X)|\rightarrow X$ defined as follows: every point $u\in |\mathcal{K}(X)|$ lies in an unique open simplex $x_1<...<x_n$, then $f(u)=x_1$. Moreover, for any continuous map $f:X\rightarrow Y$ between finite topological spaces. The map $\mathcal{K}(f):\mathcal{K}(X)\rightarrow \mathcal{K}(Y)$ given by $\mathcal{K}(f)(x_1<...<x_n)=f(x_1)\leq ...\leq f(x_n)$ is a simplicial map. Let us recall a key result in the theory of finite topological spaces (\cite[Theorem 2]{mccord1966singular}).
\begin{thm} Let $X$ be a finite topological space. Then $f_X:|\mathcal{K}(X)|\rightarrow X$ is a weak homotopy equivalence. Moreover, $\mathcal{K} $ is a functor and for every continuous map $f:X\rightarrow Y$ between finite topological spaces, $f_X\circ f=\mathcal{K}(f)\circ f_Y$.
\end{thm}
Let us recall that a weak homotopy equivalence is a continuous maps inducing isomorphisms in all homotopy groups and, consequently, in all homology groups (by a well-known theorem of J.H.C. Whitehead).
\subsection{Pro-categories and expansions}

We recall basic concepts of pro-categories. Let $\mathcal{C}$ be a category. For any directed set $\Lambda$, we say that $(X_{\lambda},p_{\lambda,\lambda'},\Lambda)$ is an inverse system in $\mathcal{C}$ whenever $X_\lambda$ is an object of $\mathcal{C}$, $p_{\lambda,\lambda'}:X_{\lambda'}\rightarrow X_{\lambda}$ is a morphism in $\mathcal{C}$ and $\lambda<\lambda'<\lambda''$ implies that $p_{\lambda,\lambda''}=p_{\lambda,\lambda'}\circ p_{\lambda',\lambda''}$. When $\Lambda=\mathbb{N}$ with the usual order we just simply say that the inverse system is an inverse sequence and omit $\Lambda$.

We first define the category inv-$\mathcal{C}$. The objects of this category are the inverse systems in $\mathcal{C}$. Given two inverse systems $(X_{\lambda},p_{\lambda,\lambda'},\Lambda)$ and $(Y_{m},q_{m,m'},M)$, a morphism $(f,f_n):(X_{\lambda},p_{\lambda,\lambda'},\Lambda)\rightarrow  (Y_{m},q_{m,m'},M)$ in this category consists of a map $f:M\rightarrow \Lambda$ and of morphisms $f_{m}:X_{f(m)}\rightarrow Y_m$ in $\mathcal{C}$, one for each $m\in M$, such that whenever $m\leq m'$, then there is a $\lambda\in \Lambda$, $f(m),f(m')\leq \lambda$, for which $f_m\circ p_{f(m),\lambda}=q_{m,m'}\circ f_{m'}\circ p_{f(m'),\lambda}$. The composition $(f,f_n)\circ (g,g_s)$, where $(f,f_n):(X_{\lambda},p_{\lambda,\lambda'},\Lambda)\rightarrow  (Y_{m},q_{m,m'},M)$ and $(g,g_s):(Y_{m},q_{m,m'},M)\rightarrow  (Z_{m},r_{m,m'},L)$, is given by $(h,h_s)$ where $h=f\circ g:L\rightarrow \Lambda $ and $h_s:X_{f(g(s))}\rightarrow Z_s$ is defined by $h_s=g_s\circ f_{g(s)}$. 

Given two morphisms $(f,f_n),(g,g_n):(X_{\lambda},p_{\lambda,\lambda'},\Lambda)\rightarrow  (Y_{m},q_{m,m'},M)$, we declare that $(f,f_n)\simeq (g,g_n)$ if and only if for every $m\in M$ there exists $\lambda\in \Lambda$, $f(m),g(m)\leq \lambda$  such that $f_m\circ p_{f(m),\lambda}=g_m\circ p_{g(m),\lambda}$. Define the category pro-$\mathcal{C}$ as the category whose objects are inverse systems in $\mathcal{C}$ and whose morphisms are the equivalence classes of this relation. For further details on inv-$\mathcal{C}$ and pro-$\mathcal{C}$, see \cite{mardevsic1982shape}.

Let $\mathcal{T}$ be a category and $\mathcal{P}$ a subcategory of $\mathcal{T}$. For an object $X$ of $\mathcal{T}$, a $\mathcal{T}$-expansion of $X$ is a morphism in pro-$\mathcal{T}$ of $X$ to an inverse system $\mathbf{X}=(X_\lambda,p_{\lambda,\lambda'},\Lambda)$ in $\mathcal{T}$,  $p:X\rightarrow \mathbf{X}$, with the following universal property: For any inverse system $Y=(Y_{\mu},q_{\mu,\mu'},M)$ in $\mathcal{P}$ and any morphism $h:X\rightarrow \mathbb{Y}$ in pro-$\mathcal{T}$, there exists a unique morphism $f:\mathbf{X}\rightarrow \mathbf{Y}$ in pro-$\mathcal{T}$ such that $h=f\circ p$. We say that $p$ is a $\mathcal{P}$-expansion of $X$ provided $\mathbf{X}$ and $f$ are in pro-$\mathcal{P}$. The subcategory $\mathcal{P}$ is dense in $\mathcal{T}$ provided every object of $\mathcal{T}$ admits a $\mathcal{P}$-expansion. These expansions play a crucial role in the definition of the abstract shape category. For more details see \cite[Chapter 1, Section 2]{mardevsic1982shape}.

\subsection{Inverse limits}

Let us consider the category of groups, denoted by $Grp$, and the topological category, denoted by $Top$. We now particularize the notion of inverse limit within these categories. Suppose that $(X_n,p_{n,n+1})$ is an inverse sequence in $Top$. Consider the product space $\prod_{n\in \mathbb{N}} X_n$ equipped with  the product topology, and let $\pi_j:\prod_{n\in \mathbb{N}} X_n\rightarrow X_j$ be the projection onto the $j$-coordinate. The inverse limit $\mathcal{X}$ of the sequence $(X_n,p_{n,n+1})$ is defined as $$\mathcal{X}=\{x\in \prod_{n\in \mathbb{N}} X_n | \pi_j(x)=p_{j,j'}(\pi_{j'}(x)) \textnormal{ for every }j\leq j'\},$$
where $\mathcal{X}$ inherits the subspace topology from $\prod_{n\in \mathbb{N}} X_n$.

Similarly, the inverse limit $\mathcal{G}$ of an inverse sequence of groups $(G_n,q_{n,n+1})$ is defined by considering the direct product $\prod_{n\in \mathbb{N}}G_n$, the canonical projection $\pi_j$ and an analogous condition to that used for $\mathcal{X}$. The group structure on $\mathcal{G}$ is inherited from its realization as a subgroup of $\prod_{n\in \mathbb{N}}G_n$. 

For a comprehensive introduction to inverse limits, we refer the reader to \cite[Chapter 1, Section 5]{mardevsic1982shape}.
  

\section{Approximation of topological algebraic invariants}\label{sec_main_result}

We begin by establishing the notation. In what follows, $Top$ denotes the category of topological spaces, while $HTop$ and $HPol$ denote the homotopy category and the homotopy category of polyhedra, respectively. Let $Y$ be a metric space and let $X \subset Y$ be a compact metric subspace, endowed with the subspace metric inherited from $Y$. Since every compact metric space can be embedded in the Hilbert cube, our assumption that $X$ is a subspace of a metric space entails no loss of generality.

\begin{figure}[h!]
\centering
\includegraphics[scale=0.7]{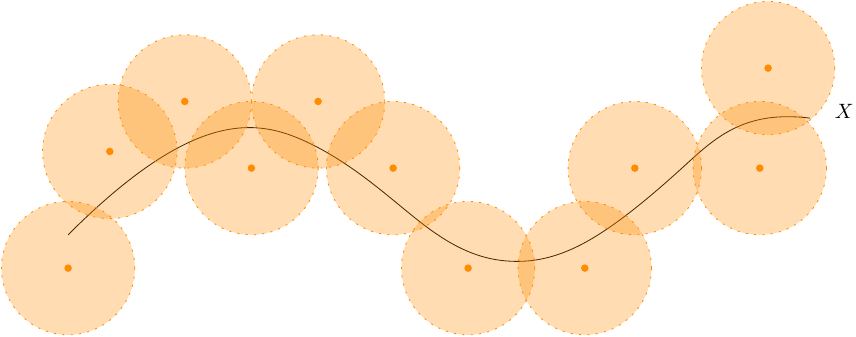}
\caption{Schematic drawing of a neighbourhood $\epsilon$-approximation of $X$.}
\label{fig_esquem_entorno_aproximacion}
\end{figure}

We say that a finite set $A \subset Y$ is a \emph{neighbourhood $\epsilon$-approximation} of $X$ if for every $x \in X$ there exists $a \in A$ such that $d(x,a) < \epsilon$, where $\epsilon$ is a positive real number (see Figure~\ref{fig_esquem_entorno_aproximacion} for a schematic representation of a neighbourhood $\epsilon$-approximation of $X$). This definition allows for the possibility that $A \cap X = \emptyset$; that is, the approximating set may contain no points of $X$. Throughout this paper and in applications, we will assume $Y = \mathbb{R}^n$ equipped with the Euclidean metric.

\begin{ex}\label{ex:lemniscate} Consider the lemniscate of Gerono $C\subset \mathbb{R}^2$, defined as the zero set of the quartic polynomial $x^4-x^2+y^2$. Set $\epsilon_n=\frac{\sqrt{5}}{2^{3n-3}}$ for each $n\in \mathbb{N}$. In Figure \ref{fig:lemnsicate_truncated}, we depict  neighbourhood $\epsilon_1$-approximation $A_1$, neighbourhood $\epsilon_2$-approximation $A_2$ and neighbourhood $\epsilon_3$-approximation $A_3$ in red, green and blue, respectively. 
\begin{figure}[h!]
\centering
\includegraphics[scale=0.27]{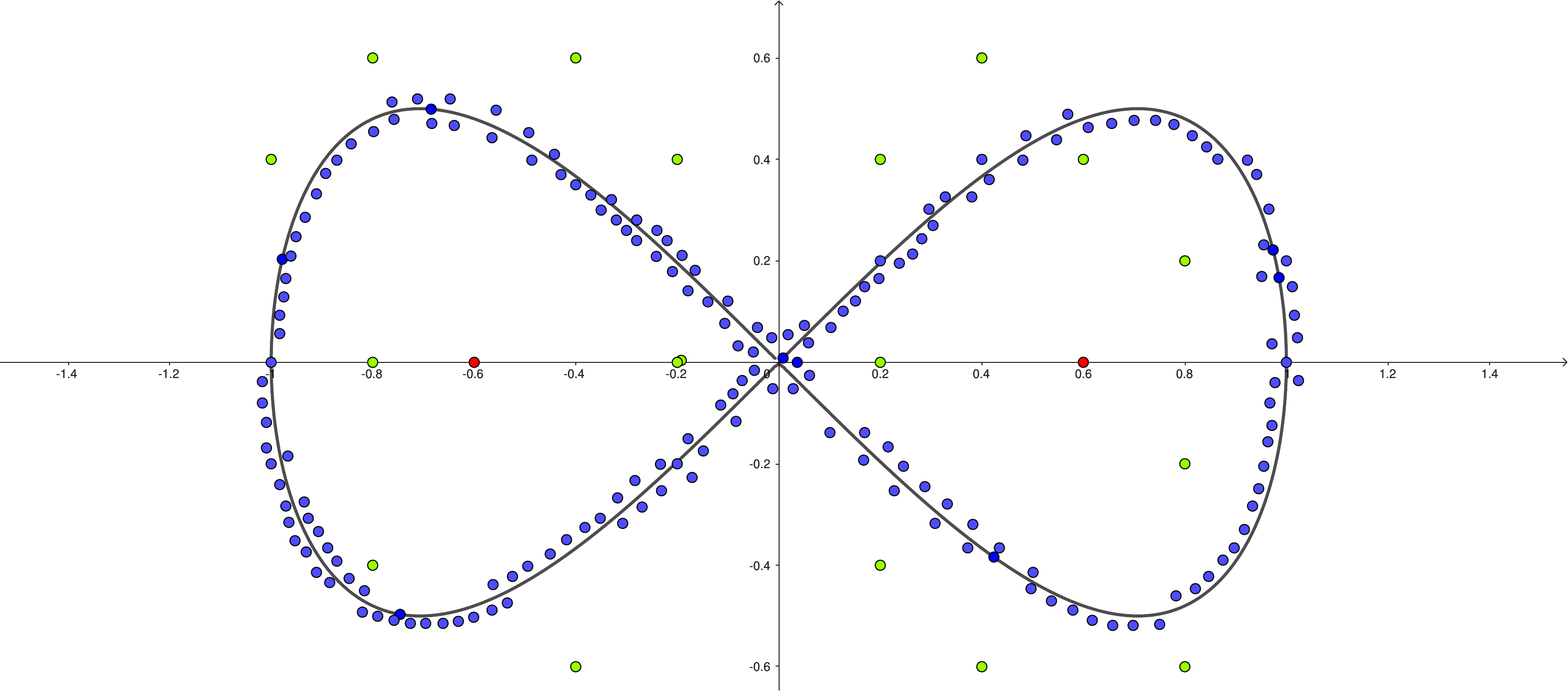}
\caption{Lemniscate of Gerono and neighbourhood $\epsilon$-approximations of it with colors.}\label{fig:lemnsicate_truncated}
\end{figure}
\end{ex}

Let $\epsilon$ be a positive real number and let $A$ be a subset of a metric space. Define $$\mathcal{U}_{\epsilon}(A)=\{C\subseteq A|\textnormal{diam}(C)<\epsilon \}.$$
Let $\{\epsilon_n\}_{n\in \mathbb{N}}$ be a sequence of positive real numbers such that $\epsilon_{n+1}< \frac{\epsilon_n}{2}$ for all $n\in \mathbb{N}$. For each natural number $n$, let $A_n$ be a neighbourhood $\epsilon_n$-approximation of $X$, and consider 
$\mathcal{U}_{4\epsilon_n}(A_n)=\{C\subseteq A_n|\textnormal{diam}(C)<4\epsilon_n \}$, with the partial order  $C\leq D$ if and only if $D\subseteq C$. It is straightforward to verify that $\mathcal{U}_{4\epsilon_n}(A_n)$ is a finite partially ordered set and thus a finite topological $T_0$-space. Define a map $q_{n,n+1}:\mathcal{U}_{4\epsilon_{n+1}}(A_{n+1}) \rightarrow \mathcal{U}_{4\epsilon_n}(A_n)$ by $q_{n,n+1}(C)=\bigcup_{c\in C}(B(c,\epsilon_n)\cap A_n)$. 
\begin{lem} The map $q_{n,n+1}:\mathcal{U}_{4\epsilon_{n+1}}(A_{n+1}) \rightarrow \mathcal{U}_{4\epsilon_n}(A_n)$ is well-defined and continuous.
\end{lem}
\begin{proof}
Let $C\in \mathcal{U}_{4\epsilon_{n+1}}(A_{n+1})$. Then $\textnormal{diam}(C)<4\epsilon_{n+1}<2\epsilon_n$. Suppose $x,y\in q_{n,n+1}(C)$. This gives that there exist $c_x,c_y\in C$ such that $d(x,c_x),d(y,c_y)< \epsilon_n$. Hence, $d(x,y)\leq d(x,c_x)+d(c_x,c_y)+d(c_y,y)<\epsilon_n+2\epsilon_n+\epsilon_n=4\epsilon_n$, which implies $q_{n,n+1}(C)\in \mathcal{U}_{4\epsilon_n}(A_n)$. Continuity follows since  $q_{n,n+1}$ preserves the order.
\end{proof}

From this, to every compact subspace $X \subset Y$ we may associate an inverse sequence $(\mathcal{U}_{4\epsilon_n}(A_n), q_{n,n+1})$ of finite topological spaces, which we will show is unique up to isomorphism in the pro-category pro-$HTop$ (Corollary~\ref{cor_unicidad}). We refer to such a sequence as a \emph{finite neighbourhood approximative sequence} of $X$. 

Similarly, we define a \emph{finite approximative sequence} as a finite neighbourhood approximative sequence $(\mathcal{U}_{4\delta_n}(B_n), q_{n,n+1})$ where $B_n \subset X$ for all $n \in \mathbb{N}$, and $\{\delta_n\}_{n\in\mathbb{N}}$ is a sequence of positive real numbers satisfying $\delta_{n+1} < \frac{\delta_n}{2}$.

Note that the main difference between these two inverse sequences lies in the location of the sampling points. In the first case, the sampling points $A_n$ need not lie in $X$, whereas in the second case the points $B_n$ are required to be contained in $X$. Hence, finite neighbourhood approximative sequences generalize finite approximative sequences.

Under suitable conditions on the sets $B_n$ and the sequence $\{ \delta_n\}_{n\in \mathbb{N}}$, the inverse limit $\mathcal{X}$ of a finite approximative sequence contains a homeomorphic copy of $X$ that is a strong deformation retract of $\mathcal{X}$ (see \cite[Theorem 4.13 and Theorem 4.15]{chocano2020computational}). This reconstruction property does not generally hold for finite neighbourhood approximative sequences. However, finite neighbourhood approximative sequences offer two key advantages:

(1) Computational efficiency. Sampling points uniformly from $X$ is often infeasible without full knowledge of $X$. In practice, classical examples such as the Warsaw Circle or the Hawaiian Earring are replaced by computational analogues to facilitate sampling (\cite{chocano2020computational,mondejar2020reconstruction}). Moreover, constructing a finite approximative sequence for the curve in Example \ref{ex:lemniscate} requires solving the equation $x^4-x^2+y^2=0$. By considering finite neighbourhood approximative sequence we avoid these situations.

(2) Approximation of algebraic invariants and shape theoretical aspects. These sequences can be used to describe shape theory (see Corollary \ref{cor_shape}) and to compute shape invariants, such as \v{C}ech homology groups. This aligns with Borsuk's original approach to shape theory, where compact metric spaces are embedded in the Hilbert cube and inverse systems of neighbourhoods are used to define the shape category (see \cite[Appendix 2]{mardevsic1982shape})



The following result is central to connecting finite neighbourhood approximative sequences with the theory of finite approximative sequences developed in \cite{chocano2020computational,chocano2025shape}. We prove it in Section \ref{sec_proof}. 

\begin{thm}\label{thm:truncated_normal} Let $X$ be a compact metric space. If $(\mathcal{U}_{4\epsilon_n}(A_n),q_{n,n+1})$ is a finite neighbourhood approximative sequence of $X$ and $(\mathcal{U}_{4\delta_n}(B_n),q_{n,n+1})$ is a finite approximative sequence of $X$, then they are isomorphic in pro-$HTop$.
\end{thm}

Theorem \ref{thm:truncated_normal} shows that finite neighbourhood approximative sequences are sufficient for studying shape-theoretic properties and for reconstructing algebraic invariants of compact metric spaces. In particular, shape theory can be described—and even defined—in terms of finite neighbourhood approximative sequences. This conclusion is supported by the results of \cite{chocano2025shape}, where shape theory is characterized using finite approximative sequences. In that paper, the authors consider a category whose objects are compact metric spaces and whose morphisms are defined by means of morphisms between their associated finite approximative sequences, viewed as objects in pro-$HTop$. They later prove that this category is isomorphic to the shape category of compact metric spaces.

\begin{cor}\label{cor_shape} Let $X$ and $Z$ be two compact metric spaces. Suppose $(\mathcal{U}_{4\epsilon_n}(A_n),q_{n,n+1})$ and $(\mathcal{U}_{4\tau_n}(C_n),q_{n,n+1})$ are finite neighbourhood approximative sequences of $X$ and $Z$, respectively. Then these sequences are isomorphic in pro-$HTop$ if and only if $X$ and $Z$ have the same shape. 
\end{cor}
\begin{proof}
Let us consider finite approximative sequences $(\mathcal{U}_{4\delta_n}(B_n),q_{n,n+1})$ and $(\mathcal{U}_{4\gamma_n}(D_n),q_{n,n+1})$ of $X$ and $Z$, respectively. By Theorem \ref{thm_homotopia_finitos}, $(\mathcal{U}_{4\delta_n}(B_n),q_{n,n+1})$ is isomorphic to $(\mathcal{U}_{4\epsilon_n}(A_n),q_{n,n+1})$ in pro-$HTop$ and $(\mathcal{U}_{4\gamma_n}(D_n),q_{n,n+1})$ is isomorphic to $(\mathcal{U}_{4\tau_n}(C_n),q_{n,n+1})$ in pro-$HTop$. Moreover, $X$ and $Z$ have the same shape if and only if  $(\mathcal{U}_{4\delta_n}(B_n),q_{n,n+1})$ and $(\mathcal{U}_{4\gamma_n}(D_n),q_{n,n+1})$ are isomorphic in pro-$HTop$ by \cite[Theorem 3.1]{chocano2025shape}. From this and the previous comment, we deduce the desired result.
\end{proof}

\begin{cor}\label{cor_reconstruct_homology} Let $X$ be a compact metric space and let $(\mathcal{U}_{4\epsilon_n}(A_n),q_{n,n+1})$ be a finite neighbourhood approximative sequence of $X$. Then the $i$-th \v{C}ech homology group $\check{H}_i(X)$ is isomorphic to the inverse limit $\varprojlim (H_i(\mathcal{U}_{4\epsilon_n}(A_n)),H_i(q_{n,n+1}))$.
\end{cor}
\begin{proof}
Let $(\mathcal{U}_{4\delta_n}(B_n),q_{n,n+1})$ be a finite approximative sequence of $X$. By \cite[Proposition 3.2]{chocano2020computational}, we have $\check{H}_i(X)\simeq \varprojlim (H_i(\mathcal{U}_{4\delta_n}(B_n)),H_i(q_{n,n+1}))$. On the other hand, by Theorem \ref{thm:truncated_normal}, we deduce that $(H_i(\mathcal{U}_{4\delta_n}(B_n)),H_i(q_{n,n+1}))$ is isomorphic to $(H_i(\mathcal{U}_{4\epsilon_n}(A_n)),H_i(q_{n,n+1}))$ in the pro-category of abelian groups, which shows the desired result.
\end{proof}

We can in fact state a stronger result:
\begin{cor}\label{cor_expansions} Let $X$ be a compact metric space and let  $(\mathcal{U}_{4\epsilon_n}(A_n),q_{n,n+1})$ be  a finite neighbourhood approximative sequences of $X$. Then $(\mathcal{K}(\mathcal{U}_{4\epsilon_n}(A_n)),\mathcal{K}(q_{n,n+1}))$ is a $HPol$-expansion of $X$.
\end{cor}
\begin{proof}
By Theorem \ref{thm:truncated_normal}, $(\mathcal{U}_{4\epsilon_n}(A_n),q_{n,n+1})$ is isomorphic to any finite approximative sequence of $X$. Consequently, $(\mathcal{K}(\mathcal{U}_{4\epsilon_n}(A_n)),\mathcal{K}(q_{n,n+1}))$ is isomorphic to a $HPol$-expansion (see the discussion of \cite[Section 2]{chocano2020computational}).
\end{proof}

For the notion of expansion and its role in shape theory, see \cite[Chapter I, Section 4]{mardevsic1982shape}. As a direct consequence of this result, other shape invariants--such as the shape dimension (see \cite[Chapter II, Section 1]{mardevsic1982shape}) and the shape groups (see \cite[Chapter II, Section 3]{mardevsic1982shape})--can be derived from the structure of the inverse sequence $(\mathcal{U}_{4\epsilon_n}(A_n),q_{n,n+1})$. 

Finally, we establish the uniqueness of the construction introduced in this section:

\begin{cor}\label{cor_unicidad}
Let $X$ be a compact metric space. Suppose that $(\mathcal{U}_{4\epsilon_n}(A_n),q_{n,n+1})$ and $(\mathcal{U}_{4\tau_n}(C_n),q_{n,n+1})$ are finite neighbourhood approximative sequences of $X$. Then the inverse sequences  $(\mathcal{U}_{4\epsilon_n}(A_n),q_{n,n+1})$ and $(\mathcal{U}_{4\tau_n}(C_n),q_{n,n+1})$ are isomorphic in pro-$HTop$.
\end{cor}
\begin{proof}
This follows immediately from Theorem \ref{thm_homotopia_finitos} and \cite[Theorem 5.1]{chocano2020computational}.
\end{proof}


\section{Method to approximate invariants and examples}\label{sec_algorithm_examples}

In this section, we present a computational approach to approximating algebraic invariants of a compact topological space $X \subset \mathbb{R}^n$, where $X$ is defined by

\[
X = \{(x_1, x_2, \ldots, x_n) \in \mathbb{R}^n \mid f(x_1, \ldots, x_n) = 0\},
\]

with $f : \mathbb{R}^n \to \mathbb{R}$ a continuous map. We restrict our study to the intersection $X \cap Q$, where $Q \subset \mathbb{R}^n$ is a fixed $n$-dimensional cube.

We propose the following iterative procedure to approximate the homological invariants of $X\cap Q$:
\begin{enumerate}
\item Initialization. Set $\epsilon_1=a>\textnormal{diam}(Q)$, choose a point $y\in Q$, and define $A_1=\{ y\}$. Then $\mathcal{U}_{4\epsilon_1}(A_1)=\{y \}$. Initialize $j=1$.

\item Sampling. Set $\epsilon_{j+1}=\frac{a}{2^{j+1}}$. Compute a neighbourhood $\epsilon_{j+1}$-approximation $A_{j+1}$ of $X\cap Q$. This can be done, for example, by:
\begin{itemize}
\item Sampling points from a uniform grid over $Q$, or
\item Drawing points from a uniform random distribution over $Q$, and retaining those sufficiently close to $X$.
\end{itemize}

\item Homology computation. Compute $$\varprojlim (H_i(\mathcal{U}_{4\epsilon_s}(A_s)),H_i(q_{s,s+1}),J),$$ where $(H_i(\mathcal{U}_{4\epsilon_s}(A_s)),H_i(q_{s,s+1}),J)$ is an inverse sequence indexed by $J=\{0,1,2,...,j+1 \}$. Set $j=j+1$ and return to Step $2$.
\end{enumerate}
To simplify computation, we may instead estimate the Betti numbers at each stage. Let $\beta_i^j=\textnormal{rank}(H_{i}(\mathcal{U}_{4\epsilon_j}(A_j) ))$ for $i=0,1,...,n-1$. We then modify Step 3 as follows: 
\begin{itemize}
\item[3.] Compute $\beta_i^{j+1}$ for each $i=0,...,m-1$.
\end{itemize}

Stopping criterion. The method may terminate under one of the following conditions:

\begin{itemize}
\item Stabilization: If $\beta_i^j=\beta_i^{j+1}=\cdots=\beta_i^{j+s}$ for some fixed $s\in \mathbb{N}$ and all $i$.
\item Iteration limit: If a maximum number of iterations is reached without stabilization.
\end{itemize}

Note that stabilization may not always occur. For example, in the case of the Cantor set, the $0$-th Betti number $\beta_0^j$ increases with $j$, reflecting the growing number of disconnected components as resolution improves.

In the following examples, we use the mathematical software R to compute the Betti numbers at various stages of the approximation process.

\begin{ex}\label{ex_lemniscata_general} Consider the set $$X=\{(x,y)\in \mathbb{R}^2|x^4-x^2+y^2=0\},$$ which defines the lemniscate of Gerono. We study $X$ within the square domain $[-3,3]\times[-3,3]$. Define the sequence $\epsilon_n=\frac{3}{2^n}$. For each $n\in \mathbb{N}$, we construct the set $$A_n=\{(a,b)| a,b\in \{-3,-3+\frac{3}{2^{n}},...,0,...,3-\frac{3}{2^{n}},3 \} \textnormal{ and } |a^4-a^2+b^2|<\frac{3}{2^n}\}.$$ 
Figure \ref{fig:four_images} illustrates the sets $A_n$ for $n=3,4,5,6$. Each $A_n$ consists of points lying within a neighbourhood of $X$. As $n$ increases, the neighbourhood becomes increasingly refined, providing a more accurate approximation of the curve.

To conclude this example, we compute the Betti numbers of $\mathcal{U}_{4\epsilon_n}(A_n)$ for the initial values of $n$. These results are presented in Figure \ref{figure_lemniscata}.  

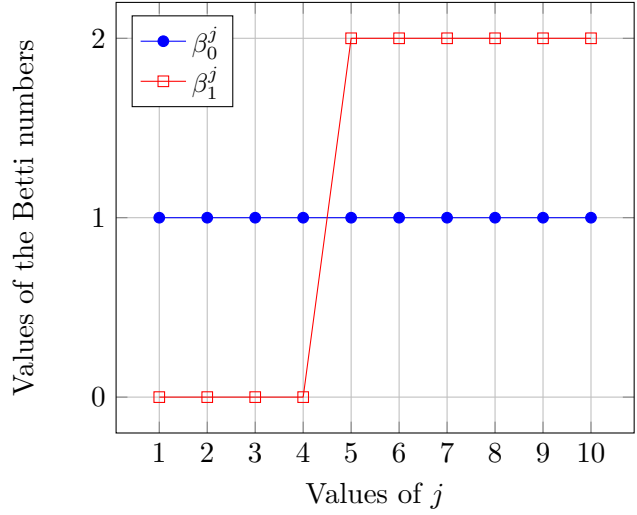
\begin{figure}[h!]
\centering
\begin{tikzpicture}
    \begin{axis}[
        xlabel={Values of $j$},
        ylabel={Values of the Betti numbers},
        grid=major,
        xtick={1,2,...,10},
        ytick={0,1,2},
        legend pos=north west,
        legend style={font=\small}
    ]
    \addplot[color=blue,mark=*] coordinates {
        (1,1) (2,1) (3,1) (4,1) (5,1) (6,1) (7,1) (8,1) (9,1) (10,1)
    };
    \addlegendentry{$\beta_0^j$}
    \addplot[color=red,mark=square] coordinates {
        (1,0) (2,0) (3,0) (4,0) (5,2) (6,2) (7,2) (8,2) (9,2) (10,2)
    };
    \addlegendentry{$\beta_1^j$}
    \end{axis}
\end{tikzpicture}
\caption{Betti numbers $\beta_i^j$ for $i=0,1$ and $j=1,...,10$ from Example \ref{ex_lemniscata_general}}\label{figure_lemniscata}
\end{figure}

\begin{figure}[ht]
    \centering
    \begin{subfigure}[b]{0.45\textwidth}
        \centering
        \includegraphics[width=\textwidth]{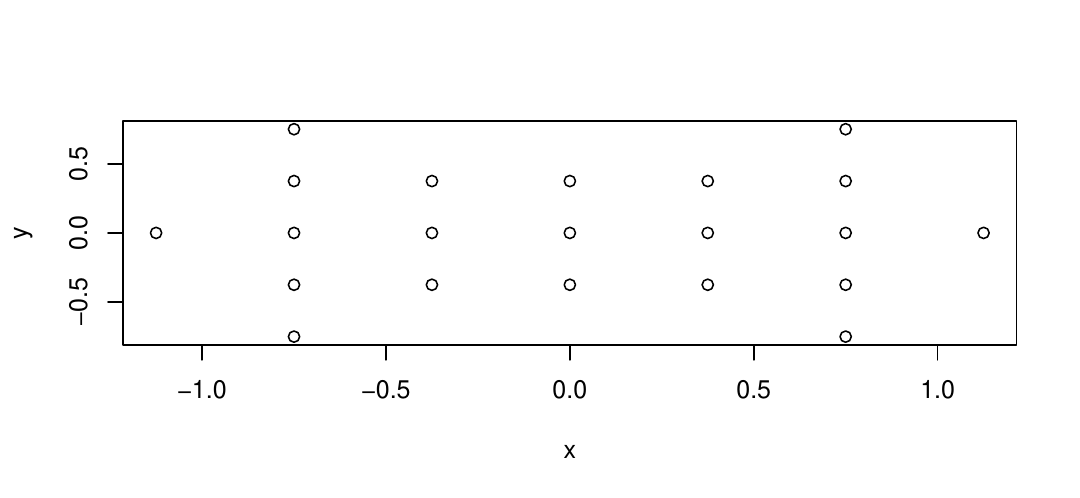}
        \caption{The set of points $A_3$ from Example \ref{ex_lemniscata_general}.}
        \label{fig:U_3_lem}
    \end{subfigure}
    \hfill
    \begin{subfigure}[b]{0.45\textwidth}
        \centering
        \includegraphics[width=\textwidth]{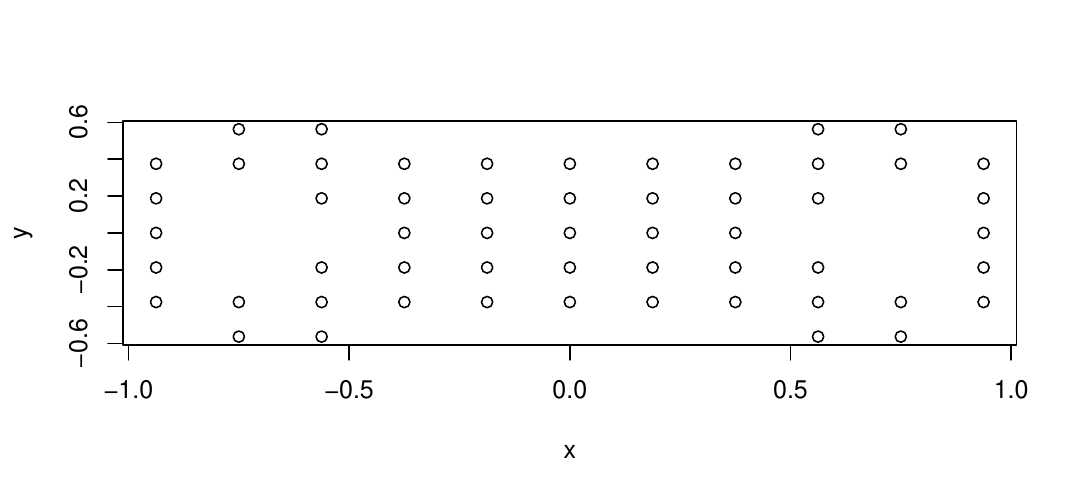}
        \caption{The set of points $A_4$  from Example \ref{ex_lemniscata_general}.}
        \label{fig:U_4_lem}
    \end{subfigure}
    \vskip\baselineskip
    \begin{subfigure}[b]{0.45\textwidth}
        \centering
        \includegraphics[width=\textwidth]{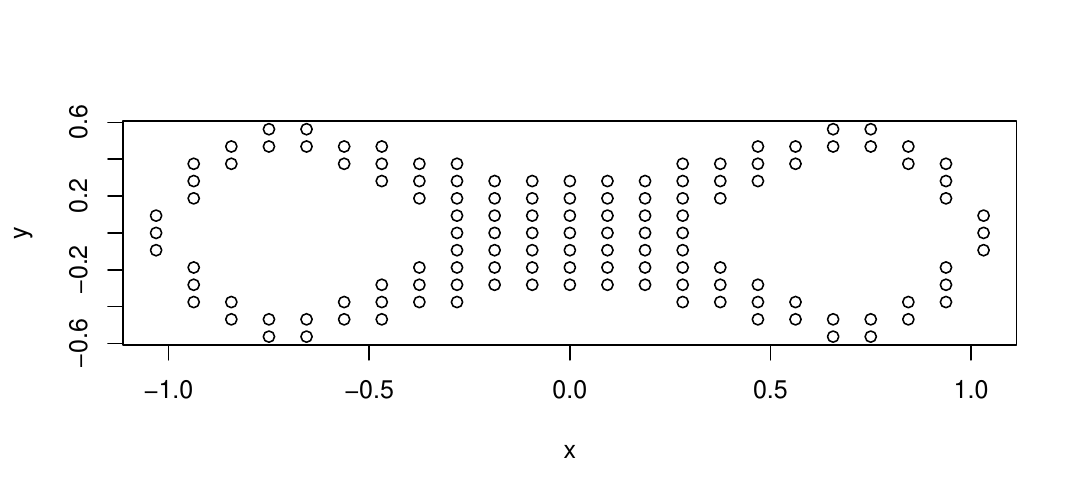}
        \caption{The set of points $A_5$ from Example \ref{ex_lemniscata_general}.}
        \label{fig:U_5_lem}
    \end{subfigure}
    \hfill
    \begin{subfigure}[b]{0.45\textwidth}
        \centering
        \includegraphics[width=\textwidth]{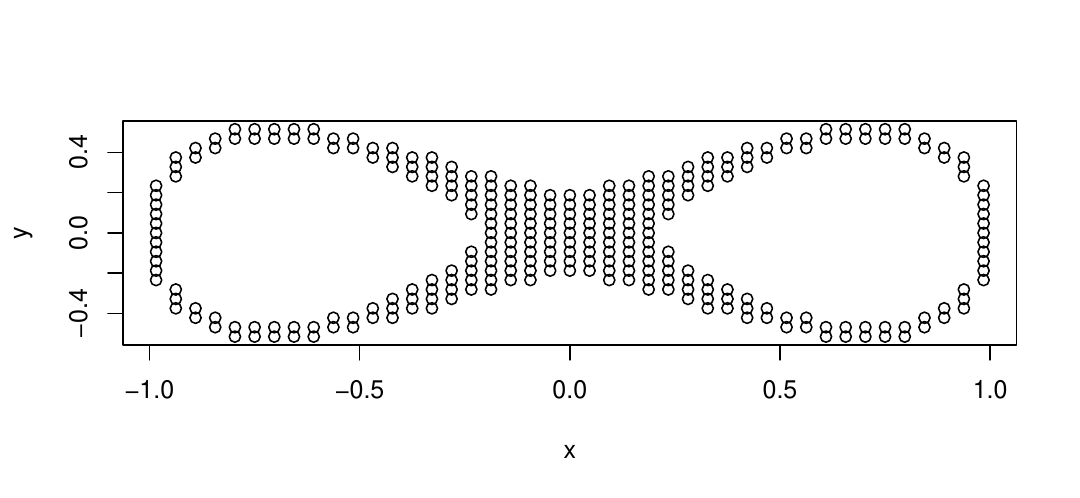}
        \caption{The set of points $A_6$  from Example \ref{ex_lemniscata_general}.}
        \label{fig:U_6_lem}
    \end{subfigure}
    \caption{Set of points $A_n$ for some values of $n$ from Example \ref{ex_lemniscata_general}.}
    \label{fig:four_images}
\end{figure}

As expected, we observe that $\beta_0^j=1$ for all $j$, indicating a single connected component. Starting from $j = 5$, the first Betti number $\beta_1 = 2$, revealing the presence of two one-dimensional holes in the sampled space.
\end{ex}

\begin{ex}\label{ex_hyper} Consider the set $$X=\{(x,y)\in \mathbb{R}^2| y^2-x^4+5x^2-4=0 \}.$$ We study $X$ within the region $[-3,3]\times[3,3]$. Define the sequence $\epsilon_n=\frac{3}{2^n}$. For each $n\in \mathbb{N}$, we construct the set $$A_n=\{(a,b)| a,b\in \{-3,-3+\frac{1}{2^{n+1}},...,0,...,3-\frac{1}{2^{n+1}},3 \} \textnormal{ and } |b^2-a^4+5a^2-4|<\frac{3}{2^n}\}.$$ We compute the Betti numbers of $\mathcal{U}_{4\epsilon_n}(A_n)$ for $n=1,...,10$. These results are summarized in Figure \ref{figure_hyper}. From this approximation, it appears that $X$ has three connected components in $[-3,3]\times [-3,3]$ and there is also one one-dimensional hole.

To localize the components, we subdivide the domain into three vertical strips:  $[-3,-2]\times [-3,3]$, $[-2,2]\times [-3,3]$ and $[2,3]\times [-3,3]$. Applying the algorithm to each subregion, we observe that:
\begin{itemize}
\item The central component, located in $[-2,2]\times [-3,3]$, contains the non-trivial first Betti number ($\beta_1$).
\item The left and right components, located in the outer strips, have $\beta_1=0$.
\end{itemize}

To support this, one may visualize the sets $A_n$ (when feasible), or apply clustering techniques to identify regions of high potential density. These methods can guide the application of the algorithm to specific subregions of interest.
\begin{figure}[h!]
\centering
\begin{tikzpicture}
    \begin{axis}[
        xlabel={Values of $j$},
        ylabel={Values of the Betti numbers},
        grid=major,
        xtick={1,2,...,10},
        ytick={0,1,2,3},
        legend pos=north west,
        legend style={font=\small}
    ]
    \addplot[color=blue,mark=*] coordinates {
        (1,1) (2,1) (3,1) (4,3) (5,3) (6,3) (7,3) (8,3) (9,3) (10,3)
    };
    \addlegendentry{$\beta_0^j$}
    \addplot[color=red,mark=square] coordinates {
        (1,0) (2,0) (3,1) (4,1) (5,1) (6,1) (7,1) (8,1) (9,1) (10,1)
    };
    \addlegendentry{$\beta_1^j$}
    \end{axis}
\end{tikzpicture}
\caption{Betti numbers $\beta_i^j$ for $i=0,1$ and $j=1,...,10$ from Example \ref{ex_hyper}}\label{figure_hyper}
\end{figure}
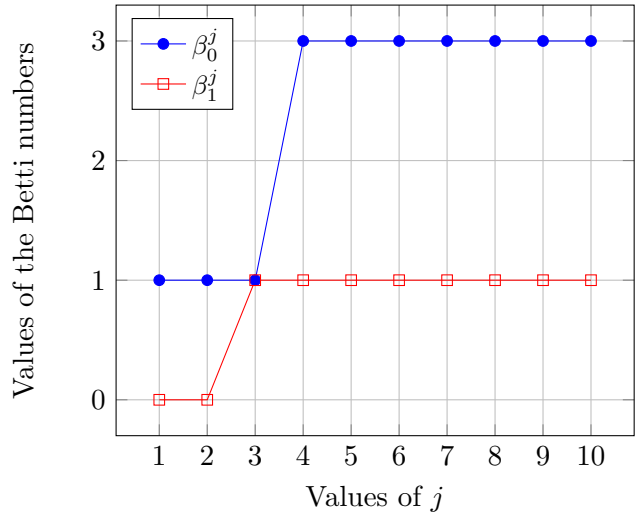
\end{ex}

To conclude, we provide a classical example from shape theory: the Hawaiian earring.

\begin{ex}\label{ex_hawaiian} Consider $$X=\bigcup_{i\in \mathbb{N}} \{(x,y)\in \mathbb{R}^2|(x-\frac{4}{2^i})^2+y^2-(\frac{4}{2^i})^2=0 \}.$$ We study $X$ within the region $[-4,4]\times[4,4]$. Define the sequence $\epsilon_n=\frac{4}{2^n}$. For each $n\in \mathbb{N}$, we construct the set \begin{align*}
 A_n=\{(a,b)| a,b\in \{-4,-4+\frac{1}{2^{n}},...,0,...,4-\frac{1}{2^{n}},4 \} \textnormal{ and }\\|(a-\frac{4}{2^i})^2+b^2-(\frac{4}{2^i})^2|<\frac{4}{2^n} \textnormal{ where } i\in \{0,1,...,n \}\}.
 \end{align*}
Note that in the construction of $A_n$, the point $(0,0)$ approximates every point in $$X\setminus \bigcup_{i=1}^n \{(x,y)\in \mathbb{R}^2|(x-\frac{4}{2^i})^2+y^2-(\frac{4}{2^i})^2=0 \}.$$ For this reason, the index $i$ in $A_n$ belongs to $\{0,1,...,n \}$. Figure \ref{fig_hawaiian} represents the Betti numbers of $\mathcal{U}_{4\epsilon_n}(A_n)$. It is evident that $X$ has only one connected component. On the other hand, as $n$ increases, $\beta_1$ also increases, which is expected given that $X$ is homeomorphic to the Hawaiian earring. 
\begin{figure}[h!]
\centering
\begin{tikzpicture}
    \begin{axis}[
        xlabel={Values of $j$},
        ylabel={Values of the Betti numbers},
        grid=major,
        xtick={1,2,...,10},
        ytick={0,1,2,...,7},
        legend pos=north west,
        legend style={font=\small}
    ]
    \addplot[color=blue,mark=*] coordinates {
        (1,1) (2,1) (3,1) (4,1) (5,1) (6,1) (7,1) (8,1) (9,1) (10,1)
    };
    \addlegendentry{$\beta_0^j$}
    \addplot[color=red,mark=square] coordinates {
        (1,0) (2,0) (3,0) (4,1) (5,2) (6,3) (7,4) (8,5) (9,6) (10,7)
    };
    \addlegendentry{$\beta_1^j$}
    \end{axis}
\end{tikzpicture}
\caption{Betti numbers $\beta_i^j$ for $i=0,1$ and $j=1,...,10$ from Example \ref{ex_hawaiian}}\label{fig_hawaiian}
\end{figure}
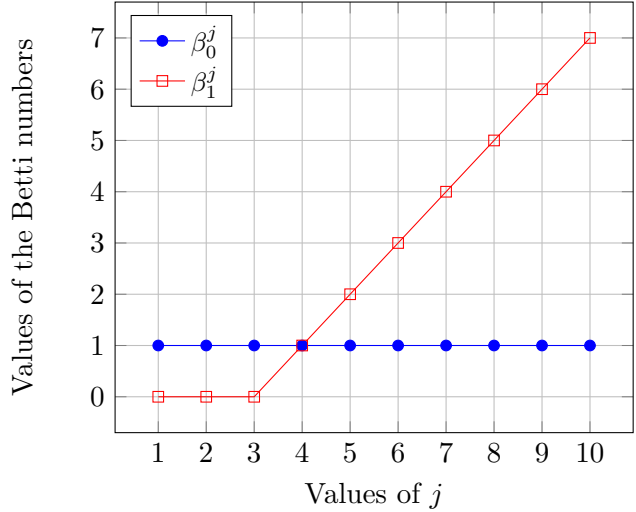
\end{ex}

\section{Proof of the Main result}\label{sec_proof}
\begin{proof}[Proof of Theorem \ref{thm:truncated_normal}]
Define $(I,I_n):(\mathcal{U}_{4\delta_n}(B_n),q_{n,n+1})\rightarrow (\mathcal{U}_{4\epsilon_n}(A_n),q_{n,n+1})$ as follows: $I(n):=\min \{ l\in \mathbb{N}| \delta_l<\frac{\epsilon_n}{16}\}$ and $I_n:\mathcal{U}_{4\delta_{I(n)}}(B_{I(n)})\rightarrow \mathcal{U}_{4\epsilon_n}(A_n)$ is given by $I_n(C):=\bigcup_{c\in C} (B(c,\epsilon_n)\cap A_n)$. Define $(T,T_n):\mathcal{U}_{4\epsilon_{T(n)}}(A_{T(n)}),q_{n,n+1})\rightarrow (\mathcal{U}_{4\delta_n}(B_n))$ as follows: $T(n):=\min \{ l\in \mathbb{N}| \epsilon_l<\frac{\delta_n}{16}\}$ and $T_n:\mathcal{U}_{4\epsilon_{T(n)}}(A_{T(n)})\rightarrow \mathcal{U}_{4\delta_n}(B_n)$ is defined by $T_n(C):=\bigcup_{c\in C} (B(c,\epsilon_{T(n)}+\delta_n)\cap A_n)$.  We need to prove that $(I,I_n)$ and $(T,T_n)$ are morphism in pro-$HTop$ and later that $(I,I_n)\circ (T,T_n)$ is homotopic to the identity morphism $(\textnormal{id},\textnormal{id}_n):(\mathcal{U}_{4\epsilon_n}(A_n),q_{n,n+1})\rightarrow (\mathcal{U}_{4\epsilon_n}(A_n),q_{n,n+1})$ and also that  $(T,T_n)\circ (I,I_n)$ is homotopic to the identity morphism $(\textnormal{id},\textnormal{id}_n):(\mathcal{U}_{4\delta_n}(B_n),q_{n,n+1})\rightarrow (\mathcal{U}_{4\delta_n}(B_n),q_{n,n+1})$. We break the proof in several steps:

\underline{$I_n$ is well defined and continuous.} Consider $C\in \mathcal{U}_{4\delta_{I(n)}}(B_{I(n)})$. Notice that $I_n(C)$ cannot be the empty set because for every $x\in X$ there exists $a\in A_n$ such that $d(x,a)<\epsilon_n$. Suppose $a_x,a_y\in I_n(C)$, which yields that there are points $x,y\in C$ satisfying $d(x,a_x),d(y,a_y)<\epsilon_n$. Then, $$d(a_x,a_y)\leq d(a_x,x)+d(x,y)+d(y,a_y)<\epsilon_n+\frac{\epsilon_n}{4}+\epsilon_n<4\epsilon_n,$$
and we conclude that $\textnormal{diam}(I_n(C))<4\epsilon_n$ and $I_n(C)\in \mathcal{U}_{4\epsilon_n}(A_n)$. By the definition of $I_n$, it is immediate that $I_n$ is continuous because it preserves the subset relation.

\underline{$(I,I_n)$ is a morphism in pro-$HTop$.} We must verify that $I_{n}\circ q_{I(n),I(m)}$ is homotopic to $q_{n,m} \circ I_m$ for every $m\geq n$. To prove this, we first define $H:\mathcal{U}_{4\delta_{I(m)}}(B_{I(m)})\rightarrow \mathcal{U}_{4\epsilon_n}(A_n)$ by $H(C):=q_{n,m}(I_m(C))\cup (I_{n}(q_{I(n),I(m)}(C)))$. We demonstrate that $H$ is well defined. Consider $C\in \mathcal{U}_{4\delta_{I(m)}}(B_{I(m)})$, and $a_x,a_y\in H(C)$ such that $a_x\in I_{n}(q_{I(n),I(m)}(C))$ and $a_y\in q_{n,m}(I_m(C))$. Then, there exist $y_a\in A_m$, $x_a\in B_{I(n)}$ and $c_x,c_y\in C$ such that $d(a_x,x_a),d(a_y,y_a)<\epsilon_n$, $d(x_a,c_x)<\delta_{I(n)}$ and $d(y_a,c_y)<\epsilon_m$. By combining these inequalities, we deduce
$$d(a_x,a_y)\leq  \epsilon_n+\delta_{I(n)}+4\delta_{I(m)}+\epsilon_m+\epsilon_n<4\epsilon_n,$$ which yields that $H$ is well defined. The continuity of $H$ follows from its construction. Clearly, we have $H\geq q_{n,m} \circ I_m,I_n \circ q_{I(n),I(m)}$. From this and Theorem \ref{thm_homotopia_finitos}, we deduce that $q_{n,m} \circ I_m$ and $I_n \circ q_{I(n),I(m)}$ are homotopic maps and we can conclude that $(I,I_n)$ is a well defined morphism in pro-$HTop$.

\underline{$T_n$ is well defined and continuous.} Consider $C\in \mathcal{U}_{4\epsilon_{T(n)}}(A_{T(n)})$. First, we prove that that $T_n(C)$ cannot be the empty set. Let $c\in C$. Then $B(c,\epsilon_{T(n)})\cap X\neq \emptyset$ by the hypothesis and we can consider $a\in B(c,\epsilon_{T(n)})\cap X$. Since $B_n$ is a $\delta_n$ approximation of $X$, there exists $b\in B_n$ such that $d(a,b)<\delta_n$. Particularly, this means that $T_n(c)=B(c,\epsilon_{T(n)}+\delta_n)\cap B_n$ is a non-empty set, which yields that $T_n(C)$ is a non-empty set. Now, suppose $b_x,b_y\in T_n(C)$, this gives that there are points $x,y\in C$ with $d(x,b_x),d(y,b_y)<\epsilon_{T(n)}+\delta_n$. Then, $$d(b_x,b_y)\leq d(b_x,x)+d(x,y)+d(y,b_y)<\epsilon_{T(n)}+\delta_n+\frac{\delta_n}{4}+\epsilon_{T(n)}+\delta_n<4\delta_n$$
and we conclude that $\textnormal{diam}(T_n(C))<4\delta_n$ and $T_n(C)\in \mathcal{U}_{4\delta_n}(B_n)$. By the definition of $T_n$, it is immediate that $T_n$ is continuous because it preserves the subset relation.

\underline{$(T,T_n)$ is a morphism in pro-$HTop$.} We prove that $T_{n}\circ q_{T(n),T(m)}$ is homotopic to $q_{n,m} \circ T_m$ for every $m\geq n$. To prove this, we first define $H:\mathcal{U}_{4\epsilon_{T(m)}}(A_{T(m)})\rightarrow \mathcal{U}_{4\delta_n}(B_n)$ by $H(C)=q_{n,m}(T_m(C))\cup (T_{n}(q_{I(n),I(m)}(C)))$. We demonstrate that $H$ is well defined. Consider $C\in \mathcal{U}_{4\epsilon_{T(m)}}(A_{T(m)})$, and $b_x,b_y\in H(C)$ such that $b_x\in q_{n,m}(T_m(C))$ and $b_y\in T_{n}(q_{T(n),T(m)}(C))$. Then, there exist $x_b\in B_m$, $y_a\in A_{T(n)}$ and $c_a,c_b\in C$ such that $d(b_x,x_b)<\delta_n$, $d(x_b,c_b)<\epsilon_{T(m)}+\delta_m$, $d(b_y,y_b)<\epsilon_{T(n)}+\delta_n$ and $d(y_b,c_b)<\epsilon_{T(n)}$. By combining these inequalities, we deduce
\begin{align*}
d(b_x,b_y)&\leq  \epsilon_{T(n)}+\delta_n+\epsilon_{T(n)}+4\epsilon_{T(m)}+\epsilon_{T(m)}+\delta_m+\delta_n\\ & <\frac{\delta_n}{16}+\delta_n+\frac{\delta_n}{16}+\frac{\delta_n}{4}+\frac{\delta_n}{16}+\delta_n+\delta_n<4\delta_n,
\end{align*}
which yields that $H$ is well defined. The continuity of $H$ follows trivially. Clearly, we have $H\geq q_{n,m} \circ T_m,q_{n,m} \circ T_m$. From this and Theorem \ref{thm_homotopia_finitos}, we deduce that $q_{n,m} \circ T_m,q_{n,m} \circ T_m$ are homotopic maps and we can conclude that $(T,T_n)$ is a well defined morphism in pro-$HTop$.

\underline{$(T,T_n)\circ (I, I_n)$ is homotopic to the identity morphism.} We verify that $q_{n,m}\circ \text{id}_m$ is homotopic to $T_n\circ I_{T(n)}\circ q_{I(T(n)),m}$ for every $m>I(T(n))$. We repeat previous arguments. Define $H:\mathcal{U}_{4\delta_{m}}(B_m)\rightarrow \mathcal{U}_{4\delta_{n}}(B_n)$ by $H(C)=q_{n,m}(C)\cup  T_n( I_{T(n)}(q_{I(T(n)),m}(C)))$. By showing that $H$ is well defined and continuous we can conclude the desired result due to Theorem \ref{thm_homotopia_finitos}. Consider $C\in \mathcal{U}_{4\delta_{m}}(B_m)$ and $x,y\in H(C)$ such that $y\in q_{n,m}(C)$ and $x\in T_n( I_{T(n)}(q_{I(T(n)),m}(C)))$. Then, there exist $a_x\in A_{T(n)}$, $b_x\in B_{I(T(n))}$ and $c_x,c_y\in C$ such that $d(y,c_y)<\delta_n$, $d(x,a_x)<\epsilon_{T(n)}+\delta_n$, $d(a_x,b_x)<\epsilon_{T(n)}$ and $d(b_y,c_y)<\delta_{I(T(n))}$. By combining these inequalities, we obtain
\begin{align*}
d(x,y)&\leq \epsilon_{T(n)}+\delta_n+\epsilon_{T(n)}+\delta_{I(T(n))}+4\delta_m+\delta_n \\
& <\frac{\delta_n}{16}+\delta_n+\frac{\delta_n}{16}+\frac{\delta_n}{256}+\frac{\delta_n}{128}+\delta_n<4\delta_n,
\end{align*}
where we are using the fact that $4\delta_m<2\delta_{I(T(n)}<\frac{\epsilon_{T(n)}}{8}<\frac{\delta_n}{128}$. 
Thus $H$ is well defined and the continuity of $H$ is trivial by its construction. By applying Theorem \ref{thm_homotopia_finitos}, we deduce that $q_{n,m}\circ \text{id}_n$ is homotopic to $T_n\circ I_{T(n)}\circ q_{I(T(n)),m}$.

\underline{$(I,I_n)\circ (T, T_n)$ is homotopic to the identity morphism.} This is equivalent to prove that $q_{n,m}\circ \text{id}_n$ is homotopic to $I_n\circ T_{I(n)}\circ q_{T(I(n)),m}$ for every $m>T(I(n))$. Define $H:\mathcal{U}_{4\epsilon_{m}}(A_m)\rightarrow \mathcal{U}_{4\epsilon_{n}}(A_n)$ by $H(C)=q_{n,m}(C)\cup  I_n( T_{I(n)}(q_{T(I(n)),m}(C)))$. Consider $C\in \mathcal{U}_{4\epsilon_{m}}(A_m)$ and $x,y\in H(C)$ with $x\in I_n( T_{I(n)}(q_{T(I(n)),m}(C)))$ and $y\in q_{n,m}(C)$. Then, there exist $b_x\in B_{I(n)}$, $a_x\in A_{T(I(n))}$ and $c_x,c_y\in C$ such that $d(x,b_x),d(y,c_y)<\epsilon_n$, $d(b_x,a_x)<\delta_{I(n)}+\epsilon_{T(I(n))} $ and $d(a_x,c_x)<\epsilon_{T(I(n))}$. Thus,
\begin{align*}
d(x,y)&\leq \epsilon_n+\delta_{I(n)}+\epsilon_{T(I(n))}+\epsilon_{T(I(n))}+4\epsilon_m+\epsilon_n\\
&<\epsilon_n+\frac{\epsilon_n}{16}+\frac{\epsilon_n}{256}+\frac{\epsilon_n}{256}+\frac{\epsilon_n}{128}+\epsilon_n<4\epsilon_n
\end{align*} 
because $4\epsilon_m<2\epsilon_{T(I(n))}<\frac{\delta_{I(n)}}{8}<\frac{\epsilon_n}{128}$. Hence, $H$ is well defined. By repeating previous arguments, we deduce that $(I,I_n)\circ (T, T_n)$ is homotopic $(\text{id},\text{id}_m)$.
\end{proof}

\bibliography{bibliografia}
\bibliographystyle{plain}

\newcommand{\Addresses}{{
  \bigskip
  \footnotesize

  \textsc{ P.J. Chocano, Departamento de Matemática Aplicada,
Ciencia e Ingeniería de los Materiales y
Tecnología Electrónica, ESCET
Universidad Rey Juan Carlos, 28933
Móstoles (Madrid), Spain}\par\nopagebreak
  \textit{E-mail address}:\texttt{ pedro.chocano@urjc.es}

}}

\Addresses

\end{document}